\documentclass[12pt]{article}
\usepackage{bbding}
\usepackage{amsmath}
\usepackage{amssymb}
\usepackage{latexsym}
\usepackage{amsthm}
\usepackage{mathrsfs}
\usepackage{fancyhdr}
\usepackage{amsfonts}
\usepackage{amssymb}
\usepackage{graphicx}
\usepackage{latexsym, bm}

\newtheorem{thm}{Theorem}[section]
\newtheorem{lem}[thm]{Lemma}
\newtheorem{cor}[thm]{Corollary}
\newtheorem{prop}[thm]{Proposition}

\newtheorem{rem}[thm]{Remark}

\newcommand{\fw}{\text{\upshape{fw}}}

\title{Face-width of Pfaffian Braces and Polyhex Graphs on Surfaces}

\author{Dong Ye\thanks{Department of Mathematical Sciences, Middle
Tennessee State University, Murfreesboro, TN 37132, USA; email: dye@mtsu.edu} \
and Heping Zhang\thanks{School of Mathematics and Statistics, Lanzhou 
University, Lanzhou, Gansu 730000, China; email: zhanghp@lzu.edu.cn}}

\date{}

\begin{document}



\maketitle

\begin{abstract}
A graph $G$ is Pfaffian if it has an orientation such that each central cycle $C$ (i.e. $C$ is even and $G-V(C)$ has a perfect matching) has an odd number of edges directed in either direction of the cycle. The number of perfect matchings of Pfaffian graphs can be computed in polynomial time. In this paper, by applying the characterization of Pfaffian braces due to Robertson, Seymour and Thomas [Ann. Math. 150 (1999) 929-975], and independently
McCuaig [Electorn. J. Combin. 11 (2004) \#R79], we
show that every embedding of a Pfaffian brace on a surface with
positive genus has face-width at most 3.  For a Pfaffian cubic brace, we obtain further structure properties which are useful in characterizing Pfaffian 
polyhex graphs.
Combining with  polyhex graphs with face-width 2, we show that  a bipartite polyhex graph is Pfaffian if and only if it is isomorphic to the cube, the Heawood graph or $C_k\times K_2$ for even integers $k\ge 6$, and all non-bipartite polyhex graphs are Pfaffian.
\end{abstract}


\section{Introduction}

Let $G$ be a graph with vertex set $V(G)$ and edge set $E(G)$. A
{\em perfect matching} of $G$ is a set $M$ of independent edges such
that every vertex of $G$ is incident with exactly one edge in $M$. A
cycle of $G$ is {\em central} if $G-V(C)$ has a perfect matching. 
If $G$ has a perfect matching, a central cycle $C$ must be of even
size; In other words, $G$ has a perfect matching $M$ such that $C$ is $M$-alternating (i.e. the edges of $C$ alternate  on and off $M$). For an orientation $D$ to $G$, an
even cycle $C$ is {\it oddly oriented} if $C$ has an odd number of edges directed in either direction of the cycle. An orientation of $G$
is {\it Pfaffian} if every central cycle of $G$ is oddly oriented. A
graph $G$ is {\it Pfaffian} if it has a Pfaffian orientation. It is
known that if $G$ has a Pfaffian orientation $D$, then the number of perfect
matching of $G$ can be obtained by computing the
determinant of the skew adjacency matrix of $D$ \cite{LP}. The Pfaffian orientation
was first introduced by Kasteleyn \cite{K1} for solving
2-dimensional Ising problem.  It is important to determine whether a given graph is Pfaffian or not. Kasteleyn \cite{K1} showed that every planar graph admits a Pfaffian orientation. Little \cite{Little} obtained a characterization for Pfaffian bipartite graph, that  a bipartite graph is Pfaffian if only if it has no matching minor isomorphic to $K_{3,3}$. The problem characterizing Pfaffian
bipartite graph is related to many interesting problems,
such as the P\'olya permanent problem, the sign-nonsingular matrix
problem, etc. But the problem of
characterizing Pfaffian non-bipartite graphs remains open. Readers may refer to  a survey of Thomas
\cite{RT} on the Pfaffian
orientations of graphs.

A connected graph $G$ is {\em $k$-extendable} ($|V(G)|\ge 2k+2$) if $G$ has
$k$ independent edges and each set of $k$ independent edges $G$ is contained in  a perfect matching. A $k$-extendable graph is $(k-1)$-extendable and
$(k+1)$-connected \cite{P1,LP}. A 2-extendable bipartite graph is
also called a {\em brace}. So a brace is 3-connected. A 3-connected bicritical graph $G$
(the deletion of  any pair of distinct  vertices of $G$ results in  graph with a perfect
matching)  is called a {\em brick}.  Every bicritical graph is 1-extendable and non-bipartite. By
Lov\'asz's tight-cut decomposition \cite{L,LP}, every 1-extendable
graph can be reduced to a list of braces and bricks. Vazirani and
Yannakakis \cite{VY} showed that a graph $G$ is Pfaffian if and only
if all braces and bricks generated from tight-cut decomposition are Pfaffian. Robertson, Seymour and Thomas \cite{RST}, and independently
McCuaig \cite{WM} presented an elegant
construction of Pfaffian braces, which leads to a polynomial time algorithm to determine
whether a given braces is Pfaffian or not.

Given an embedding
$\Pi$ of a graph $G$ on a surface $\Sigma$,  a closed simple curve $\ell$
in $\Sigma$ is {\em
contractible} if $\Sigma-\ell$ has precisely two components and at
least one is homeomorphic to an open disk. The {\em face-width} (or
{\em representativity}) of  $\Pi$  is the maximum integer $k$ such that every non-contractible
simple closed curve in  $\Sigma$ intersects the graph at
least $k$ points (see \cite{MT}), denoted by $\text{fw}(G,\Pi)$. 
For convienence, assume that a plane graph has face-width infinity.  Robertson and Vitray
\cite{RV90}, and independently Thomassen \cite{T90}, showed that a
planar graph embedded on a surface $\Sigma$ with positive genus
has the face-width at most 2.

In this paper, we consider the face-width of Pfaffian braces on surfaces.  By using the tri-sum operations of Pfaffian braces in an extensive sense, in Section 2 we mainly show that Pfaffian brace embedded on a surface  with positive genus
has the face-width at most 3. It is natural to ask whether  Pfaffian bricks
have this property. In the end of this paper we give a negative answer to this question via non-bipartite polyhex graphs.

In Section 3, applying Robertson et al. and  McCuaig's constructions of Pfaffian braces \cite{RST,WM}, we obtained detailed structure properties of Pfaffian cubic 
braces, which are useful for characterizing of Pfaffian polyhex 
graphs.

A {\em polyhex graph} is a cubic graph cellularly embedded on a
surface such that every face is bounded by a hexagon. By Euler's
formula, the surface could be only the torus and the Klein bottle.
It is known that a bipartite polyhex graph is cubic brace \cite{YZ}. 
Polyhex graphs have been considered as surface tilings
\cite{Neg,T91}. A detailed classification of polyhex graphs was
given by Thomassen \cite{T91}. Polyhex graphs are also considered as
a possible generalization of  fullerenes  \cite{DFRR}  in chemistry and material science
\cite{KI,KMP}. P.E. John
\cite{J} tried using the Pfaffian method to enumerate the perfect
matchings of polyhex graphs on torus. However, not all polyhex
graphs on torus are Pfaffian.

In Section 4, we give the construction of polyhex graphs and determine which polyhex graphs have face-width 2. In Section 5, we completely characterize Pfaffian polyhex graphs. We show that a bipartite polyhex graph is Pfaffian if and only if it is planar or is isomorphic to the Heawood graph, and all non-bipartite polyhex graphs are Pfaffian.

\section{Face-width of Pfaffian braces on surfaces}

In this section, we mainly show that the face-width of Pfaffian braces on surfaces with positive genus is at most 3.

Let $G$ be a graph {\em cellularly} embedded on a surface $\Sigma$ (each face is homeomorphic to an open disk in the plane). For a face $f$
of $G$, the boundary (or $f$ itself ) is often represented by a closed walk of $G$,  denoted by $\partial f$, which is called {\em facial walk}.  An embedding $\Pi$ is a {\em strong embedding} if every facial walk is a cycle. The {\em face-width} 
$\fw(G,\Pi)$ of an embedding $\Pi$ is the smallest number $k$ such that there 
exist $k$ faces whose union contains a non-contractible curve.
Undefined notations
and concepts are referred to \cite{MT}.

The following result presents an important property of a planar
graph embedded on a surface $\Sigma$ with genus $g(\Sigma)\ge 1$.

\begin{thm}[\cite{RV90,T90}]\label{thm2-1}
Let $G$ be a planar 3-connected graph. Then every embedding of $G$
on a surface $\Sigma$ with $g(\Sigma)>0$ has face-width at most 2.
\end{thm}

Kasteleyn \cite{K2} showed that every planar graph is Pfaffian. But
the above result does not hold for all Pfaffian graphs. For example,
consider the Heawood graph that is a
Pfaffian brace (a Pfaffian orientation is shown in Figure \ref{fig2-1}). The Heawood graph
admits an embedding on the torus with face-width three (to be shown
in Section 4). 

\begin{figure}[!hbtp]\refstepcounter{figure}\label{fig2-1}
\begin{center}
\includegraphics{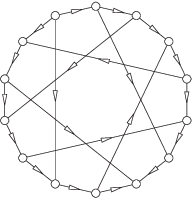}\\
{Figure \ref{fig2-1}: The Heawood graph admitted   a Pfaffian
orientation.}
\end{center}
\end{figure}

\begin{lem}[Prop. 5.5.12 on Page 150 in \cite{MT}]\label{lem:face-width3}
Let $G$ be a graph and $\Pi$ be an embedding of $G$ on a surface $\Sigma$. Then the following
conditions are equivalent:\\
(1) $\fw (G)\ge 3$ and $G$ is 3-connected;\\
(2) All facial walks of $\Pi$ are cycles and any two of them are either disjoint
or their intersection is just a vertex or an edge.
\end{lem}

\begin{lem}\label{heawood}\label{lem2-2}
Every embedding of the Heawood graph on a surface $\Sigma$ has
face-width at most 3.
\end{lem}
\begin{proof}
Let $\Pi$ be an embedding of the Heawood graph $G$ on a surface
$\Sigma$. Suppose to the contrary that $\fw(G, \Pi)=k\ge 4$. 
Then $\Sigma$ has a non-contractible closed curve passing through faces  $f_0,f_1,f_2, \ldots, f_{k-1}$  of $\Pi$ in turn such that $k$ is minimum. Since
$G$ is 3-connected, 
every $\partial f_i$ is a cycle by Lemma~\ref{lem:face-width3}.
By minimality of $k$ and Lemma~\ref{lem:face-width3},  distinct faces $f_i$ and $f_{j}$ have an edge in common  if and only if $|j-i|\equiv 1 \pmod k$.
Since $G$ has the girth  six, $|V(\partial f_i)|\ge 6$. So
\[14=|V(G)| \ge|V(\bigcup_{i=0}^{k-1} \partial f_i)|\ge 6k-2k=4k\ge 16,\]
a contradiction. That implies 
 $k\le 3$.
\end{proof}

We will show that the above result holds for  all Pfaffian braces.

\begin{thm}\label{thm:2-3}
Let $G$ be a  Pfaffian brace. Then every embedding of $G$
on a surface $\Sigma$ with $g(\Sigma)>0$ has face-width at most 3.
\end{thm}

In order to prove Theorem \ref{thm:2-3}, we need the characterization
of Pfaffian braces obtained by Robertson, Thomas and Seymour
\cite{RST}, and independently by McCuaig \cite{WM}. Let $G_0$ be a
graph and $C$ a central 4-cycle of $G_0$; that means, $G_0-V(C)$ has a
perfect matching. Let $G_1$, $G_2$ and $G_3$ be three subgraphs of
$G_0$ such that $G_1\cup C_2\cup G_{3}=G_0$, and for distinct
$i,j\in \{1,2,3\}$,
$G_i\cap G_j=C$ and $V(G_{i})-V(C)\ne
\emptyset$.
A graph $G$ is a {\em tri-sum} of $G_1$, $G_2$ and $G_3$
if it is obtained from $G_0$ by deleting some edges (possibly none)
of $C$. For example, see Figure \ref{fig2-2}. 
 By a result of McCuaig (Lemma 19 on page 36 in  \cite{WM}), every
tri-sum of three braces is a new brace. 
The following is a variant of the statement of the characterization
of Pfaffian braces obtained in \cite{WM,RST}. 

\begin{figure}[!hbtp]\refstepcounter{figure}\label{fig2-2}
\begin{center}
\includegraphics{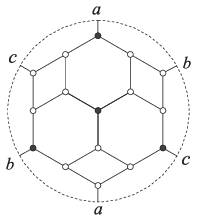}\\
{Figure \ref{fig2-2}: A cubic brace on the projective plane generated from three copies of
the cube by tri-sum.}
\end{center}
\end{figure}

\begin{thm} [\cite{RST,WM}, Theorem 4.2 in \cite{RT}] \label{thm:2-4}
A bipartite graph is a
Pfaffian brace if and only if it is isomorphic
to the Heawood graph, or it can be obtained from planar braces by
repeated application of the tri-sum operation.
\end{thm}

Let $G$ be a graph, and $G_1$ and $G_2$ be two subgraphs of $G$. 
Denote the set of edges joining vertices of $G_1$ and vertices of $G_2$
by $E(G_1,G_2)$. For a vertex $w$ of $G$, the set of all neighbors of $w$
in $G$ is denoted by $N(w)$.
The following technical lemma will be used to prove Theorem~\ref{thm:2-3}. 

\begin{lem}\label{lem:2-5}
Let $G$ be a 3-connected graph and $W$ be  a vertex-cut of size 4. If $G-W$ can be decomposed into
three disjoint graphs $G_1, G_2$ and $G_3$ such that for distinct $i, j\in \{1,2,3\}$
and $w\in W$,
$E(G_i, G_j)=\emptyset$ and $N(w)\cap V(G_i)\ne \emptyset$,
then every embedding of $G$ on a surface with positive genus has  face-width at most 3.
\end{lem}

\begin{proof} Let $\Pi$ be an embedding of $G$ on a surface $\Sigma$ with $g(\Sigma)>0$.
Suppose to the contrary that $\text{fw}(G,\Pi)>3$.

Let $W=\{v_1,v_2,v_3,v_4\}$ and $E(W)$ be the set of all edges with two
end-vertices in $W$.
For any $v_i\in W$, let $v_iv_{i,0}, ..., v_iv_{i,k_i}$ be the edges
incident with $v_i$ in clockwise direction in some small
neighborhood of $v_i$ homeomorphic to an open disc in the plane. Note that $N(v_i) \cap V(G_{\alpha})\ne \emptyset$ for every $\alpha\in \{1,2,3\}$. Let
$f_{i,\alpha}$ be a
face containing $v_iv_{i,\alpha_i}$ and $v_iv_{i,\alpha_i+1}$ such that $v_{i,\alpha_i}\in V(G_{\alpha})$ and
$v_{i,\alpha_i+1}\notin V(G_{\alpha})$.
 For an edge $v_iv_{i,k}\in E(W)$, let $f$ be the face containing $v_iv_{i,k}$ and $v_iv_{i,k+1}$.
By the definition of faces $f_{i,\alpha}$ and the ordering of 
edges incident with $v_i$, $f\ne f_{i,\alpha}$ for any $\alpha\in \{1,2,3\}$
because of $v_{i,k}\notin G-W$. 
So the intersection of $f_{i,\alpha}$ and $ f_{i,\beta}$ does not contain edges  of
$E(W)$ which are
incident with $v_i$ if $f_{i,\alpha}\ne f_{i,\beta}$. 

\medskip

(1) {\sl The intersection of two distinct faces $f_{i,\alpha}$  and $f_{i,\beta}$ does
not contain another vertex from $W$ different from $v_i$.}

\medskip

Note that $G$ is 3-connected and $\fw(G, \Pi)>4$. By Lemma~\ref{lem:face-width3},
the intersection of $f_{i,\alpha}$ and $f_{i,\beta}$ is either $v_i$ or an edge $v_iv_{i,k}$.
Since $f_{i,\alpha}$ and $f_{i,\beta}$ do not
contain an edge from $E(W)$ incident with $v_i$, it follows that $v_{i,k}\notin W$. So (1) holds.

\medskip

(2) {\sl Every $f_{i,\alpha}$ contains precisely two vertices of
$W$.}\medskip

First, we show that every $f_{i,\alpha}$ contains at least two
vertices of $W$.
By Lemma~\ref{lem:face-width3}, each $f_{i,\alpha}$ is 
bounded by a cycle. Note that
$v_{i,\alpha_i} \in V(G_{\alpha})$, $v_{i,\alpha_i+1}\notin
V(G_{\alpha})$.  If $v_{i,\alpha_i+1}\in W$, then $f_{i,\alpha}$ contains
two vertices from $W$. So suppose that $v_{i,\alpha_i+1}\notin W$. 
Hence, $\partial f_{i,\alpha}-W$ consists of at
least two components: one contains $v_{i,\alpha_i}$ and the other
contains $v_{i,\alpha_i+1}$. These two components are joined by
vertices from $W$. So
$f_{i,\alpha}$ contains at least
two vertices from $W$. 

Since $N(v_i)\cap V(G_{\alpha})\ne \emptyset$, there are
at least three distinct faces $f_{i,\alpha}$ ($\alpha\in \{1,2,3\}$) incident
with $v_i$. If one of them contains three vertices from $W$, then there
eixsts a pair of vertices which contained by two distinct faces $f_{i,\alpha}$
and $f_{i,\beta}$, contradicting (1). So (2) holds.
\medskip

By (1) and (2), we further have the following property:\medskip

{\sl Any two vertices of $W$ are contained by some face
$f_{i,\alpha}$, and every face $f_{i,\alpha}$ contains precisely a
pair vertices of $W$.}\medskip

Without loss of generality, assume that $f_{1,1}$ contains
$v_1$ and $v_2$, and $f_{1,2}$ contains $v_1$ and $v_4$, and
$f_{1,3}$ contains $v_1$ and $v_3$ (relabeling $G_{\alpha}$ for
$\alpha=1,2,3$ if necessary). Then, assume that $v_2$ and $v_4$ are
contained by $f_{2,\alpha}$, $v_2$ and $v_3$ are contained by
$f_{2,\beta}$ with $\beta\ne \alpha$,  $v_3$ and
$v_4$ are contained by $f_{3,\mu}$ for some $\mu\in \{1,2,3\}$. (For
example, see Figure \ref{fig2-3} (left). The shadow parts illustrate
the regions of $\Sigma$ containing vertices and edges from only one
of $G_{\alpha}-W$ ($\alpha\in \{1,2,3\}$), or one edge from
$E(W)$.) Then each of $f_{1,1}\cup f_{2,\beta}\cup f_{1,3}$,
$f_{1,1}\cup f_{2,\alpha}\cup f_{1,2}$, $f_{1,2}\cup f_{3,\mu}\cup
f_{1,3}$ and $f_{2,\alpha}\cup f_{2,\beta}\cup f_{3,\mu}$ contains a
closed curve which intersects $G$ at three vertices of $W$. Denote
these closed curves by $\ell_1$, $\ell_2$, $\ell_3$ and $\ell_4$.
\vspace{0.3cm}

\begin{figure}[!hbtp]\refstepcounter{figure}\label{fig2-3}
\begin{center}
\includegraphics{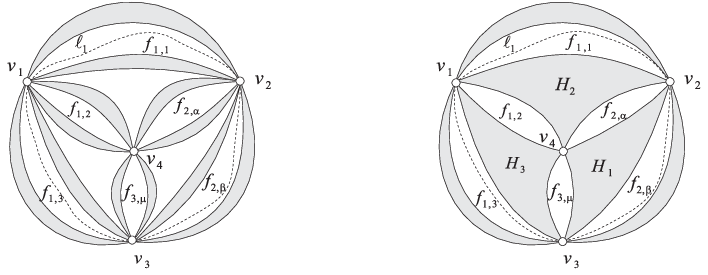}\\
{Figure \ref{fig2-3}: Illustration for the proof of Theorem
\ref{thm:2-3}.}
\end{center}
\end{figure}

(3) {\sl At least one of the closed curves $\ell_i$ $(i=1,...,4)$ is
non-contractible.} \vspace{0.3cm}

If not, suppose that all $\ell_i$ ($i=1,...,4)$ are contractible.
Then $\ell_i$ separates $\Sigma$ into two regions, and at least one
of them is homeomorphic to an open disc, denoted by $D_i$.

First,
suppose that $v_4$ lies on $D_1$. Let $R_1\subseteq D_1$ be the
region bounded by faces $f_{1,1}$, $f_{1,2}$ and $f_{2,\alpha}$.

By the definition of $f_{1,2}$, $R_1$ contains a vertex from $G_2$.
Let $H_2$ be the subgraph of $G$ (not including
$v_1, v_2$ and $v_4$) embedded in
$R_1$. By (1), every face of the embedding of $H_2$ inherited from 
the embdding $\Pi$ of $G$ in the region $R_1$ contains at most one vertex
from $v_1, v_2$ and $v_4$. So $H_2\subseteq G_2$.
Similarly, the subgraph $H_3$ (not including $v_1,v_3$ and $v_4$)
in the region $R_2\subseteq D_1$ bounded by faces $f_{1,2}$,
$f_{1,3}$ and $f_{3,\mu}$ is a subgraph of $G_{3}$. Further,
the subgraph $H_1$ (not including $v_2,v_3$ and $v_4$) in the region $R_3\subseteq D_1$ bounded by faces
$f_{2,\alpha}$, $f_{2,\beta}$ and $f_{3,\mu}$ is a subgraph
of $G_{1}$ (see Figure \ref{fig2-3} (right)).

Let $G'$ be the subgraph  (including $v_1, v_2$ and $v_3$) in $\Sigma-D_1$.
If $V(G')=\{v_1,v_2,v_3\}$, then $G$ is embedded in an open disc of $\Sigma$.
Hence $\text{fw}(G,\Pi)=0$ because $g(\Sigma)>0$.

So suppose that $G'$ contains at least one vertex
$w$ different from $v_1, v_2$ and $v_3$. Let
$H_{\alpha}'=G'\cap G_{\alpha}$. Then all vertices $v_1, v_2$ and $v_3$
have a neighbor in $H_{\alpha}'$ since $G$ is 3-connected. If two
of $H_1', H_2'$ and $H_3'$  are not empty,
then there exists a face $f_{1,\gamma}$
for some $\gamma\in \{1, 2,3\}$ which contains a vertex $v_{i,\gamma_i}\in
H_{\gamma}'$ and a vertex $v_{i,\gamma_i+1}\notin H_{\gamma}'$.
By (1), $f_{1,\gamma}$ contains two vertices from $W$. Then
there are two distinct faces, one $f_{1,\gamma}$ and another from
$f_{1,1}, f_{1,2}, f_{1,3}$ contains the same pair of vertices from $W$,
a contradiction.
So only one of $H_1', H_2'$ and $H_3'$ is not empty. Without
loss of generality, assume $H_1'\ne \emptyset$ and hence
$G'-W\subset G_1$.

Then $G_2=H_2$. Note that $N(v_3)\cap V(G_2)=\emptyset$, contradicting
that $N(v_i)\cap V(G_{\alpha})\ne \emptyset$ for any $v_i\in W$ and any
$\alpha\in \{1,2,3\}$. The contradiction implies that $v_4$ does not lie on $D_1$.

By symmetry, the region $D_i$ ($i\in \{1,2,3,4\}$) does not contain
the vertex of $W$ which is not on $\ell_i$. Then $\Sigma$ must be the
sphere since it is formed by pasting the four disc $D_i$ along
the four closed curves $\ell_i$, contradicting
$g(\Sigma)>0$.\medskip

 (3) implies that $\text{fw}(G,\Pi)\le 3$, contradicting $\fw (G, \Pi)>3$. The
contradiction completes the
proof. 
\end{proof}

A minimal vertex-cut $W$ of a graph $G$ is a {\em tri-cut} if $G-W$
has at least three components. For any vertex $w$ of a tri-cut $W$ and
any component $G_i$ of $G-W$, $N(w)\cap V(G_i)\ne \emptyset$. As a
direct corollary of Lemma~\ref{lem:2-5}, we have the following result.

\begin{thm}
Let $G$ be a 3-connected graph with a tri-cut of size 4. Then every embedding of $G$ on a surface
with positive genus has face-width at most 3.
\end{thm}

Now we are going to prove Theorem~\ref{thm:2-3}.\medskip

\noindent {\bf Proof of Theorem~\ref{thm:2-3}:} Let $G$ be a Pfaffian
brace. If $G$ is the Heawood graph or a planar brace, the result follows
from Lemma~\ref{lem2-2} and Theorem~\ref{thm2-1}.
Otherwise, by Theorem~\ref{thm:2-4}, $G$ is generated from planar Pfaffian braces by
applying tri-sum operations.
Assume that $G$ is generated from Pfaffian braces $G_1$, $G_2$ and $G_3$ by the tri-sum
operation along a central cycle $C=v_1v_2v_3v_4v_1$. Since each $G_i$ is
 3-connected, $N(v_j)\cap (V(G_i)-V(C))\ne \emptyset$ for every $v_j\in V(C)$.
Note that $G$ itself is a brace and hence 3-connected. So $G$ and $W=V(C)$ satisfy 
the conditions of Lemma~\ref{lem:2-5}. 
Hence every embedding of $G$ on a surface with positive genus has face-width at
most 3 by Lemma~\ref{lem:2-5}. \qed
\medskip

\section{Pfaffian cubic braces}

In order to characterize Pfaffian polyhex graphs, we need more structure properties
on Pfaffian cubic braces. The following result is a construction for Pfaffian
cubic braces which follows from Theorem~\ref{thm:2-4}.

\begin{thm} \label{thm2-6}
A bipartite cubic graph is a Pfaffian brace  if and only if it is
isomorphic to the Heawood graph, or it can be obtained from planar
cubic braces by repeated application of the tri-sum operation.
\end{thm}
\begin{proof}If $G$ is a non-planar Pfaffian cubic brace different from the Heawood graph,
then by Theorem \ref{thm:2-4}, $G$ is generated from planar braces by tri-sum
operations. We may assume that $G$ is generated from three 
Pfaffian braces $G_1$,
$G_2$ and $G_3$ along a 4-cycle $C$ by a tri-sum
operation. Since each $G_i$  is 3-connected and $G$ is cubic,  each
vertex of $C$ has precisely one neighbor in each $G_i-V(C$). 
So every $G_i$ is  cubic. The other direction follows directly 
from Theorem \ref{thm:2-4}.\end{proof}

In the following, we present some properties of Pfaffian cubic
braces which are useful in characterizing Pfaffian ployhex graphs.

A set $S$ of edges of a graph $G$ is a {\em cyclic edge-cut} if $G-S$ has
two components each of which contains a cycle. If every cyclic edge-cut of $G$ has
at least $k$ edges, $G$ is said to be {\em cyclically
k-edge-connected}. The {\em cyclic edge-connectivity} of $G$ is the
maximum integer $k$ such that $G$ is cyclically $k$-edge-connected,
denoted by $c\lambda(G)$.

\begin{thm}[\cite{HP,HP2}]\label{thm2-5}
Let $G$ be a cubic bipartite graph. Then $G$ is a brace if and only
if it is cyclically 4-edge-connected.
\end{thm}

\begin{thm}\label{thm2-7}
Let $G$ be a Pfaffian cubic brace different from the Heawood graph.
Then $c\lambda (G)=4$. Further,  for any 4-cycle $C$ of $G$, $E(C, G-V(C))$ 
is a cyclic 4-edge cut.
\end{thm}

\begin{proof} First we claim that if $G$ has a 4-cycle $C$, then $E':=E(G-V(C),C)$ forms a cyclic 4-edge cut.  

It is obvious that $|E'|=4$. Further, $G$ has at least 
$n\ge 8$ vertices; Otherwise $G$ is isomorphic to $K_{3,3}$, which is not Pfaffian, a contradiction.  Since $G$ is 3-connected, $G-V(C)$ is connected. Hence $G-V(C)$  has at least one cycle, because it has $n-4$ vertices and $3n/2-8\ge
n-4$ edges ($n\ge 8$). That is,   $E'$ is a cyclic 4-edge-cut and the claim holds.

If $G$ is planar,  $G$ must contain a 4-cycle since $G$ is bipartite. The Claim implies that $G$ contains a cyclic 4-edge cut. If $G$ is  non-planar,  then $G$ can be  generated from Pfaffian cubic braces $G_1$, $G_2$, and $G_3$ by the
tri-sum along a 4-cycle $C$. Since $G$ is also different from the Heawood graph, the Claim implies that each $E(C, G_i-V(C))$ is a cyclic 4-edge cut of $G_i$, which is also a cyclic 4-edge cut of $G$ since each $G_i-V(C)$ contains a cycle. So $c\lambda(G)\le 4$. Hence the
theorem follows from Theorem \ref{thm2-5}.
\end{proof}

Let $G$ be a 3-connected graph with a tri-cut $W$ of size 4. Then $W$ is called an {\em ideal tri-cut}  if $W$ is independent in $G$ and $G-W$ has exactly three components $G_1',G_2'$ and $G_3'$ such that each $G_i'$ has at least four vertices, and each $E(W, G_i')$ is a matching saturating all vertices of $W$.

\begin{prop}\label{prop2-8}
Let $G$ be a Pfaffian cubic brace that is non-planar and is different from the Heawood graph. Then $G$ has an ideal tri-cut $W$.
\end{prop}

\begin{proof} By Theorem \ref{thm2-6}, $G$ is generated  from three Pfaffian cubic braces $G_1$, $G_2$ and $G_3$ by a tri-sum operation along a 4-cycle $C$. Then $W:=V(C)$ is independent in $G$. Since each $G_i$ is 3-connected and Pfaffian, $G_i-V(C)$ is connected and has at least four vertices. Since each
vertex of $W$ has a neighbor in each $G_i-W$, $W$ is a minimal vertex-cut of $G$. It remains to show that each $E(W, G_i-W)$ is a matching saturating all vertices of $W$. This holds because $E(W, G_i-W)$ is a minimum cyclic edge-cut of $G_i$ (see Theorem \ref{thm2-7}).
\end{proof}

\section{Construction of polyhex graphs}

A polyhex graph is a cubic graph embedded on a surface such that
every face is a hexagon, a cycle of length six. So a polyhex graph
is a strong embedding. 
By Euler's formula, the
surfaces can be only the torus and the Klein bottle \cite{DFRR}.

Two graphs $G_1$ and $G_2$ are isomorphic, denoted by
$G_1\cong G_2$,  if there is a bijection
$\sigma: V(G_1)\to V(G_2)$ such that $uv\in E(G_1)$ if and only if
$\sigma(u)\sigma(v)\in E(G_2)$, and such $\sigma$ is an {\em isomorphism}  between
$G_1$ and $G_2$. For two polyhex graphs $G_1$ and $G_2$, an
isomorphism $\sigma$ from $G_1$ to $G_2$ is {\em hexagon-preserving}
if  $h\subseteq G_1$ is a hexagon if and only if  $\sigma(h)$ is also a hexagon of
$G_2$. An {\em isomorphism} from a graph $G$ to itself is called an
{\em automorphism}. A graph $G$ is {\em vertex-transitive} if, for
any two vertices $v_1,v_2\in V(G)$, there exists an automorphism
$\sigma$ such that $\sigma(v_1)=v_2$. A polyhex graph $G$ is {\em
hexagon-transitive} if for any two hexagons $h_1$ and $h_2$, there
exists a hexagon-preserving automorphism $\sigma$ such that $\sigma(h_1)=h_2$.

\begin{figure}[!hbtp]\refstepcounter{figure}\label{fig3-1}
\begin{center}
\includegraphics[scale=1.4]{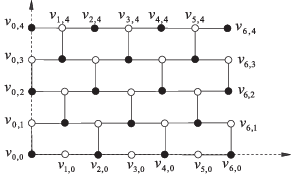}\\
{Figure \ref{fig3-1}: A rectangular hexagon lattice $L(6,4)$. }
\end{center}
\end{figure}

Take a rectangle $R$ on  a 2-dimensional Euclidean plane $\mathbb R^2$: $R=\{(x,y): 0\leq x\leq k, 0\leq y\leq q\}$, where $k$ and $q$ are  positive integers. Let $v_{i,j}$
be a vertex corresponding to the point $(i,j)$ where $i$ and $j$ are
non-negative integers. A rectangular hexagon lattice $L(k,q)$
is a graph on $R$ consisting of all vertices $v_{i,j}$ in $R$  and edges in
$\{v_{i,j}v_{i+1,j}|0\le j\le q, 0\le i\le k-1\}\cup
\{v_{i,j}v_{i,j+1}| 0\le i\le k, 0\le j\le q-1 \mbox{ and } i\equiv
j \mod 2\}$. For example, $L(6,4)$ is shown in Figure \ref{fig3-1}.

The vertices  $v_{i,j}$ of a $L(k,q)$ are colored in black or white according as  $i+j$ is even or odd. For even $k$, a polyhex tube $L'(k,q)$ is obtained from $L(k,q)$ by identifying
the vertices $v_{0,j}$ and $v_{k,j}$ for $j=0,1,\ldots,q$. So $L(k,q)$ and $L'(k,q)$ are  bipartite graphs. The cycle
$v_{0,i}v_{1,i}\cdots v_{k,i}$ is  called the {\em $i$-th
layer}, denoted by $L_i$. Let
$h_{i,j}$ denote the hexagon with the center $(2i+\alpha_j,
j+\frac{1} 2)$ where 
\[
\alpha_j:=\left\{
 \begin{array}{ll}
0,     &\mbox{if $j\equiv 1\pmod 2$;}\\
1,   &\mbox{if $j\equiv 0\pmod 2$.}
 \end{array}
 \right.
\] 
Equivalently,
\[
h_{i,j}:=\left\{
 \begin{array}{ll}
v_{2i-1,j}v_{2i,j}v_{2i+1,j}v_{2 i+1,j+1}v_{2i,j+1}v_{2i-1,j+1}v_{2i-1,j},     &\mbox{if $j\equiv 1\pmod 2$;}\\
v_{2i,j}v_{2i+1,j}v_{2i+2,j}v_{2i+2,j+1}v_{2i+1,j+1}v_{2i,j+1}v_{2i,j},   &\mbox{if $j\equiv 0\pmod 2$.}
 \end{array}
 \right.
\]

Altschuler \cite{A} showed that polyhex graphs on torus, denoted by $T(k,q,t)$,  can be constructed in the following way: From $L'(k,q)$ ($k$ even)
identify the vertices $v_{i,0}$ and $v_{i+q+2t,q}$ where $0\le t\le
k/2-1$ and the first subscript is always modulo $k$; that is, $v_{i,0}$ is connected to $v_{i+q+2t,q-1}$ by an edge for each odd $i$. For example, see Fig. \ref{fig3-2}.

\begin{figure}[!hbtp]\refstepcounter{figure}\label{fig3-2}
\begin{center}
\includegraphics[scale=1.3]{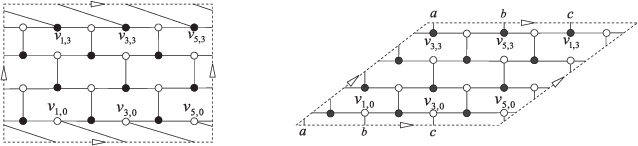}\\
{Figure \ref{fig3-2}: Representations for the polyhex graph
$T(6,4,0)$ on the torus. }
\end{center}
\end{figure}

Since a polyhex graph on the torus is a strong
embedding, we have the following result.

\begin{thm}\label{thm3-1}
A polyhex graph on the torus is isomorphic to
$T(k,q,t)$ for $k\equiv 0 \pmod 2$ and $(k,q,t)\notin \{(2,q,t),
(4,1,t), (k,1,0), (k,1,k/2-1)|k,q,t\in \mathbb Z, 0\le t\le
k/2-1\}$.
\end{thm}

Thomassen \cite{T91} classified polyhex graphs into seven types: two on the torus, five on the Klein bottle. According to Thomassen's constructions, Li et al. \cite{LLZ} reclassified the polyhex graphs on the Klein bottle into the following two types:

\begin{itemize}
\item Bipartite polyhex
$K_e(k,q)$($q\ge 2$, $k\ge 4$ is even): From $L'(k,q)$,
identify $v_{i,0}$ with $v_{k-i,q}$ if $q$ is even, and $v_{i,0}$ with $v_{k-i-1,q}$ if $q$ is odd.

\item Non-bipartite polyhex  $K_o(k,q)$ ($q\ge 2$ is even, and $k\ge 3$): From $L(k,q)$ first identify
$v_{i,0}$ with $v_{i,q}$;
then identify $v_{0,j}$ with $v_{k,q-1-j}$
if $k$ is even, and $v_{0,j}$ with $v_{k,q-j}$ if
$k$ is odd.
\end{itemize}

\begin{figure}[!hbtp]\refstepcounter{figure}\label{fig3-3}
\begin{center}
\includegraphics[scale=1.3]{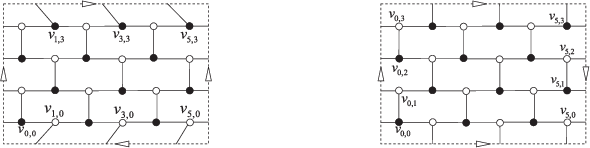}\\
{Figure \ref{fig3-3}: Polyhex graphs $K_e(6,4)$ (left) and
$K_o(6,4)$ (right) on the Klein bottle. }
\end{center}
\end{figure}

\begin{thm}\cite{LLZ}\label{thm3-1-2}
A polyhex graph on the Klein bottle is isomorphic to
either $K_e(k,q)$ ($q\ge 2$, even $k\ge 4$) or $K_o(k,q)$ (even $q\ge 2$,
$k\ge 3$).
\end{thm}

Polyhex graphs contains many interesting graphs. For example,
$T(14,1,2)$ is isomorphic to the Heawood graph, both $T(8,1,1)$ and
$K_e(4,2)$ are isomorphic to the cube $Q_3$,  and $T(6,1,1)$ is
isomorphic to $K_{3,3}$.

\begin{prop}[\cite{ZY}]\label{prop3-2}
There is a hexagon-preserving isomorphism between $T(k,q,t)$ and
$T(k,q,t')$ where $t'\equiv (k-2q-2t)/2$ $(\mbox{\upshape{mod} }
k/2)$.
\end{prop}

\begin{prop}[\cite{SLZ}]\label{prop3-3}
Every polyhex graph on the torus is vertex-transitive and
hexagon-transitive.
\end{prop}

\begin{thm}[\cite{YZ}]\label{thm3-4}
Every bipartite polyhex graph is a brace.
\end{thm}

By Theorems \ref{thm2-5} and \ref{thm3-4}, we have the following
result.

\begin{cor}\label{cor3-5}
Every bipartite polyhex graph is cyclically 4-edge-connected.
\end{cor}

\begin{lem}\label{lem3-6}
Let $G$ be a polyhex graph on the torus. Then {\upshape
$\text{fw}(G)= 2$} if and only if $G$ is isomorphic to $T(k,q,t)$
where integers $k,q,t$ satisfy $(k,q,t)\in \{(4,q,t)|q\ge 2\}$ or
$(k,q,t)\in \{(k,2,t)|k\ge 6, t\in \{k/2-2, k/2-1,0\}\}$ or
$(k,q,t)\in \{(k,1,t)|k\ge 6, k/4-1\le t\le k/4\}$.
\end{lem}

\begin{proof}
Let $G$ be a polyhex graph on the torus. Then $G$ is isomorphic to
some $T(k,q,t)$ by Theorem \ref{thm3-1}. It follows that $\text{fw}(G)= 2$ if and only if  $G$
has two distinct hexagons $h_1$ and $h_2$ which intersect in two edges. By
Proposition \ref{prop3-3}, without loss of generality let $h_1$ be the hexagon $h_{0,0}$ and let $h_2=h_{x,y}$  with $0\leq y\leq q-1$. Since $h_1$ and $h_2$ intersect in two edges,
$y=$0 or 1.

If $y=1$, then $q=2$. If follows that $x=1$ or $k-1$ since
$k>4$. Hence $t\in \{k/2-2, k/2-1,0\}$. Conversely, a polyhex
$T(k,2,t)$ with $t\in \{k/2-2, k/2-1,0\}$ does have face-width 2.

If $y=0$, then either $k=4$ or $q=1$.  If $k=4$,  then $q\ge 2$ by Theorem~\ref{thm3-1}, and further $\text{fw}(T(4,q,t))=2$. The lemma 
holds. 
So suppose $q=1$.   Then $v_{i,0}=v_{i+1+2t,1}$ for each $i$.  
As $h_1$ and $h_2$ intersect in two edges,  $\{v_{0,0}, v_{2,0}\}\cap
V(h_{x,0})\ne\emptyset$ and $\{v_{2x,0}, v_{2x+2,0}\}\cap
V(h_{0,0})\ne \emptyset$. Note that $\{v_{0,0},v_{2,0}\}\cap V(h_{x,0})\ne \emptyset$ implies
\begin{equation}2x-1\le 1+2t\le 2x+1,\end{equation} and $\{v_{2x,0}, v_{2x+2,0}\}\cap
V(h_{0,0})\ne \emptyset$ implies 
\begin{equation}  k-1\le 2x+1+2t \le k+1.\end{equation}
Combining inequalities (1) and (2), we have 
 $ k/4-1\le t\le k/4$. 
 Conversely, if $k/4-1\le t\le k/4$ and $k\ge 6$, let $x=\lfloor k/4\rfloor$. It follows that $h_{0,0}$ and $h_{x,0}$ intersect two edges. So $T(k,1,t)$ with $k/4-1\le t\le k/4$ and $k\ge 6$ has face-width 2. 
\end{proof}

By Lemmas \ref{heawood} and \ref{lem3-6}, the Heawood graph has an embedding $T(14,1,2)$ on the
torus with face-width 3. \medskip

\begin{thm}\label{thm:4-8}
Let $G$ be a polyhex graph on the Klein bottle. Then
{\upshape$\text{fw}(G)=\min\{\lceil k/2\rceil, q\}.$}
\end{thm}

\begin{proof}
Let $G$ be a polyhex graph on the Klein bottle. By Theorem \ref{thm3-1-2}, 
$G$ is isomorphic to either $K_e(k,q)$ or $K_o(k,q)$.

Let $\mathcal H$ be a set of
hexgons of $G$ such that $\bigcup\limits_{h_{x,y}\in \mathcal H} h_{x,y}$ contains a non-contractible curve $\ell$ and $\mathcal H$ is minimum. Then $\text{fw}(G)
=|\mathcal H|$. 

First, assume that $G$ is isomorphic to $K_e(k,q)$ ($k$ is even). Let $H_i$ be the graph consisting of all hexagons $h_{x,y}$ with $y=i$.
If $\mathcal H\cap H_i\ne \emptyset$ for all $i\in \mathbb Z_k$, then $|\mathcal H|\ge q$.
So suppose that $\mathcal H\cap H_{t}=\emptyset$ for some $t\in \mathbb Z_k$. Let 
$E_t=\{v_{i,t}v_{i, t+1}| i+t\equiv 0 \pmod 2, i\in \mathbb Z_k\}$. All hexagons of $K_e(k,q)-E_t$ induce a tube. 
Let $R_{j}$ be the graph consisting of all hexagons $h_{m,y}$ with $y\in \mathbb Z_q\backslash\{t\}$,
where 
\[
m =\left\{
 \begin{array}{ll}
 j     &\mbox{if $0\le y\le t-1$;}\\
 k/2-j   &\mbox{if $t+1\le y\le q-1$.}
 \end{array}
 \right.
\] Denote 
$E^j=\{e=h_{m,y}\cap h_{m,y+1}| y,y+1\in \mathbb Z_q\backslash\{t\}\}
\cup\{e_1,e_2\}$ where $e_1=h_{j,t-1}\cap h_{j,t}$ and $e_2=h_{k/2-j,t}\cap h_{k/2-j,t+1}$. Note that $E^j\subset E(R_j)$.
Then the union of all hexagons of $K_e(k,q)-E_t-E^j$ does not contain a non-contractible closed curve.
So $\mathcal H\cap R_j\ne \emptyset$ for any $j\in \mathbb Z_{k/2}$. Hence $|\mathcal H|\ge k/2$.
It follows that $\fw(G)=|\mathcal H|\ge \min\{q, k/2\}$.

Conversely, each $H_{i}$ and $R_{\lceil k/4\rceil}$ of $K_e(k,q)$ contains a non-contratible closed curve. 
Note that $H_i$ has $k/2$ hexagons and $R_{\lceil k/4\rceil}$ has $q$ hexagons. 
It follows that $\fw(G)\le \min\{q,k/2\}$. Hence $\fw(G)= \min\{q,k/2\}$.

In the following, assume that $G$ is isomorphic to $K_o(k,q)$ ($q$ is even). 
Let $R_i$ be the graph consisting of all hexagons $h_{x,y}$ with $x=i$.
If $\mathcal H\cap R_i\ne \emptyset$ for all integer $i\in [0,\lceil k/2\rceil]$, then $|\mathcal H|\ge \lceil k/2\rceil$.
So suppose that $\mathcal H\cap R_t=\emptyset$ for some $t\in \mathbb Z_{\lceil k/2\rceil}$. 
Denote $E^t=\{v_{2t,j}v_{2t+1,j}|j\in \mathbb Z_q\}$. All hexagons of $K_o(k,q)-E^t$ 
induce a tube. Let $H_j$ be the graph consisting of all hexagons $h_{x,m}$ with $x\in 
\mathbb Z_{\lceil k/2\rceil\backslash\{t\}}$, where 
\[
m =\left\{
 \begin{array}{ll}
j     &\mbox{if $0\le x\le t$;}\\
q-1-j   &\mbox{if $t+1\le x\le  k/2-1$ and $k$ is even;}\\
q-j &\mbox{if $t+1\le x\le  \lceil k/2\rceil-1$ and $k$ is odd;}
 \end{array}
 \right.
\]
Let $E_j=\{v_{i,m}v_{i,m+1}| i+m\equiv 0 \pmod 2, i\in \mathbb Z_{k}\}$ where $m$ is defined as above.
Then the union of all hexagons of $K_o(k,q)-E^t-E_j$ does not contain
a non-contractible closed curve. Note that the edges in $E_j$ are contained by only hexagons
from $H_j$. So $\mathcal H\cap H_j\ne \emptyset$ 
for any $j\in \mathbb Z_{q}$. Hence $|\mathcal H|\ge q$. It follows that 
$\fw(G)=|\mathcal H|\ge \min\{q, \lceil k/2\rceil\}$. 

On the other hand, 
both $H_{q/2}$ and each $R_{i}$ of $K_o(k,q)$
contain a non-contractible closed curve. Note that $H_{q/2}$ has $\lceil \frac k 2\rceil$ 
hexagons, and $R_{\lceil k/2 \rceil-1}$ contains $q$ hexagons. So 
$\text{fw}(G)\le \min\{\lceil \frac k 2\rceil, q\}$. This completes the proof.
\end{proof}

\begin{lem}\label{lem3-10}
Let $G$ be a polyhex graph on the Klein bottle. Then
{\upshape$\text{fw}(G)=2$} if and only if $G$ is isomorphic to
either $K_e(k,q)$ with $k=4$ or $q=2$, or $K_o(k,q)$ with $3\le k\le
4$ or $q=2$.
\end{lem}

\begin{proof} By Theorem \ref{thm:4-8}, $\text{fw}(G)=2$ if and only if 
$\min\{\lceil \frac k 2\rceil, q\}=2$. It follows that $3\le k\le 4$ or $q=2$.

By Theorem \ref{thm3-1-2}, $G$ is isomorphic to either $K_e(k,q)$ with $k=4$ or $q=2$, or $K_o(k,q)$ with $3\le k\le
4$ or $q=2$.
\end{proof}

\section{Pfaffian polyhex graphs}

Now we are ready to characterize Pfaffian polyhex graphs.

\subsection{Polyhex graphs on the torus}

Let $G$ be a polyhex graph on the torus. If $G$ is planar, then it
is Pfaffian. For a non-planar polyhex graph $G$ not isomorphic to
the Heawood graph, $G$ must contain a tri-cut by Proposition
\ref{prop2-8} if it is Pfaffian.

\begin{lem}\label{lem4-1}
Let $G$ be a polyhex graph on the torus. Then $G$ does not contain
an ideal tri-cut.
\end{lem}
\begin{proof}
Let $G$ be a polyhex graph on the torus. By Theorem \ref{thm3-1}, $G$ can be represented as  $T(k,q,t)$ for
some  suitable triple of integers $(k,q,t)$. Suppose to the contrary
that $G$ has an ideal tri-cut $W$. Then $|W|=4$ and $G-W$ has exactly three components, denoted by $G_1$, $G_2$ and $G_3$, each of which has at least four vertices.

First suppose $q\ge 2$. If $T(k,q,t)$ has a layer $L_i$ containing at least three vertices of $W$. Then $k\ge 6$ as $W$ is independent in $G$. Since $q\ge 2$,  any other layer $L_j$ ($j\ne i$)
contains at most one vertex in $W$. So all $L_j-W (j\ne i)$ are contained in a common
component of $G-W$. Note that at least one vertex of each component of
$L_i-W$ has a neighbor in either $L_{i+1}-W$ or $L_{i-1}-W$. Hence $G-W$ is connected, a contradiction.
So  every layer $L_i$ has at most two vertices in  $W$. That implies that
$L_i-W$ has at most two paths as components. So $G-W$ has at most
two components, also a contradiction.

Now suppose $q=1$.  By
Lemma \ref{prop3-3}, let
$W=\{v_{i_0,0},v_{i_1,0},v_{i_2,0},v_{i_3,0}\}$ with  $0=i_0<i_1<i_2<i_3\le 2k-2$. Then $L_0-W$ has four paths $P_0$, $P_1$,
$P_2$ and $P_3$ such that $v_{i_{\alpha},0}$ joins $P_{\alpha}$ and $P_{\alpha+1} \mbox{ (the subscripts modulo  4})$. Since $G-W$ has three
components,   precisely two paths
from $P_{\alpha}$'s ($\alpha\in \mathbb Z_4$) belong to a common
component of $G-W$.  Since each $E(W, G_i)$ is a matching, each $P_i$ has at least two vertices and $P_i$ and $P_{i+1}$ can not belong
to a common component.  So we may assume that $P_1$ and $P_{3}$ belong to a common component of $G-W$. Then $P_0$ and $P_2$ themselves induce components of $G-W$.

Let $v_{m,0}\in V(P_2)$ be adjacent to
$v_{0,0}\in W$. Then $m=k-(1+2t)$ since
$v_{m,0}=v_{m+1+2t,1}=v_{0,1}$.  Note that $v_{m+2,0}=v_{k+1-2t,0}$ is adjacent
to $v_{2,0}\in V(P_1)$. Since $P_1$ and $P_2$ belong to different components of $G-W$, $v_{m+2,0}\notin V(P_2)$.  Since each $E(W, G_i)$ is a matching and $P_1$ and $P_3$ belong to the same component of $G-W$, $v_{m+2,0}\not=v_{i_2,0}$ and thus $v_{m+2,0}\in V(P_3)$. Then $v_{m+1,0}=v_{i_2,0}$, and $v_{m-2,0}\in V(P_2)$ is adjacent to $v_{k-2,0}\in V(P_0)$, which contradicts that $P_2$ and $P_0$ belong to different components of $G-W$.
\end{proof}

By Proposition~\ref{prop2-8}, Theorem~\ref{thm3-1} and Lemma~\ref{lem4-1}, we immediately
have the following result.

\begin{cor}\label{thm4-2}
Let $G$ be a polyhex graph on the torus. Then $G$ is Pfaffian if and
only if it is planar or isomorphic to the Heawood graph $T(14,1,2)$.
\end{cor}

By Theorem \ref{thm2-1}, planar polyhex graphs embedded on the torus
must have face-width two. Using Kuratowski's Theorem that a graph is
planar if and only if it contains no $K_5$ and $K_{3,3}$ as minor,
the following lemma characterizes planar polyhex graphs on the
torus.

\begin{lem}\label{lem4-3}
A polyhex graph on the torus is planar if and only if it is
isomorphic to either $T(4,2,t)$, or $T(8,1,t)$ or $T(k,2,k/2-1)$.
\end{lem}
\begin{proof} We can see that both $T(4,2,t)$ and $T(8,1,t)$ are isomorphic
to the cube $Q_3$ and hence are planar. For $T(k,2,k/2-1)$, it is isomorphic to $C_k\times K_2$ (a
plane embedding shown in Figure \ref{fig4-1}).

\begin{figure}[!hbtp]\refstepcounter{figure}\label{fig4-1}
\begin{center}
\includegraphics[scale=1.2]{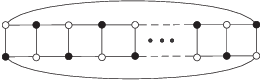}\\
{Figure \ref{fig4-1}: A plane embedding of $T(k,2,k/2-1)$.}
\end{center}
\end{figure}

Next we will show that the other polyhex graphs on the torus
are not planar. By Theorem \ref{thm2-1} and Lemma \ref{lem3-6}, it
suffices to show that $T(4,q,t)$ ($q\ge 3$), and $T(k,2,t)$ ($k\ge 6$, $t=0$
or $k/2-2$), and $T(k,1,t)$ ($k\ge 6$, $k\ne 8$ and $k/4-1\le t\le k/4$) are
all non-planar.

\begin{figure}[!hbtp]\refstepcounter{figure}\label{fig4-2}
\begin{center}
\includegraphics[scale=1.3]{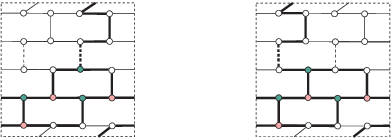}\\
{Figure \ref{fig4-2}: $K_{3,3}$-subdivisions in $T(4,q,t)$ with
$q\ge 3$.}
\end{center}
\end{figure}

Note that $T(4,q,t)$ ($q\ge 3$) is non-planar by Kuratowaski's
Theorem since it contains a subdivision of $K_{3,3}$ as a subgraph
(see the subgraphs induced by thick lines in Figure \ref{fig4-2}).

\begin{figure}[!hbtp]\refstepcounter{figure}\label{fig4-3}
\begin{center}
\includegraphics[scale=1.3]{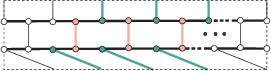}\\
{Figure \ref{fig4-3}: A $K_{3,3}$-minor in $T(k,2,0)$ ($k\ge 6$).}
\end{center}
\end{figure}

By Proposition \ref{prop3-2}, $T(k,2,0)$ is isomorphic to
$T(k,2,k/2-2)$. Since $T(k,2,0)$ ($k\ge 3$) contains a
$K_{3,3}$-minor (see Figure \ref{fig4-3}), both $T(k,2,0)$ and
$T(k,2,k/2-2)$ are non-planar.

\begin{figure}[!hbtp]\refstepcounter{figure}\label{fig4-4}
\begin{center}
\includegraphics[scale=1.3]{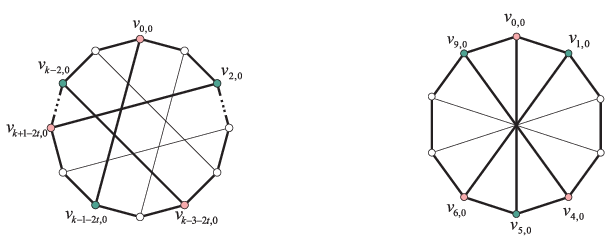}\\
{Figure \ref{fig4-4}: $K_{3,3}$-subdivisions in $T(k,1,t)$ ($k\ge
12$ and $k/4-1\le t\le k/4$ (left)) and $T(10,1,2)$ (right).}
\end{center}
\end{figure}

Now consider $T(k,1,t)$ ($k\ge 6$, $k\ne 8$ and $k/4-1\le t\le k/4$). Since
$T(6,1,1)$ is isomorphic to $K_{3,3}$, it is thus  non-planar. For
$T(k,1,t)$ with $k\ge 12$ and $k/4-1\le t\le k/4$, it contains a
subdivision of $K_{3,3}$ as a subgraph induced by edges
$v_{0,0}v_{k-1-2t,0}, v_{2,0}v_{k+1-2t},v_{k-1,0}v_{v-3-2t,0}$ and
all edges of $L_0$ (see  Figure \ref{fig4-4} (left)). If
$k=10$,  $t=2$. So
$T(10,1,2)$ contains a subdivision of $K_{3,3}$ as shown in Figure
\ref{fig4-4} (right). \end{proof}

\begin{thm}\label{thm4-4}
Let $G$ be a polyhex graph on the torus. Then $G$ has a Pfaffian
orientation if and only if $G$ is isomorphic to either the Heawood graph,or the cube $Q_3$, or $C_k\times K_2$ for even $k\ge 4$.
\end{thm}

\subsection{Polyhex graphs on the Klein bottle}

First, we consider bipartite polyhex graphs on the Klein bottle.

\begin{lem}\label{lem4-5}
Let $G$ be a bipartite polyhex graph on the Klein bottle. Then $G$
does not contain an ideal tri-cut.
\end{lem}

\begin{proof} By Theorem \ref{thm3-1-2}, we have that $G$ is isomorphic to
$K_e(k,q)$ with $q\ge 2$. An analogous argument as the proof of Theorem
\ref{lem4-1} in case  $q\ge 2$ shows that the
lemma is true.
\end{proof}

It can be seen that the Heawood graph can not be a polyhex graph on the Klein bottle. So, by Lemma \ref{lem4-5} and Proposition \ref{prop2-8}, a Pfaffian
bipartite polyhex graph on the Klein bottle must be planar.

\begin{lem}\label{lem4-6}
A bipartite polyhex graph $G$ on the Klein bottle is planar if and
only if it is $K_e(4,2)$.
\end{lem}
\begin{proof} Let $G$ be a planar bipartite polyhex graph on the Klein
bottle. By Theorem \ref{thm2-1} and Lemma \ref{lem3-10}, we may
assume that $G$ is isomorphic to one of $K_e(4,q)$ and $K_e(k,2)$ ($q\ge 2$ and $k\ge 4$).

\begin{figure}[!hbtp]\refstepcounter{figure}\label{fig4-5}
\begin{center}
\includegraphics[scale=1.3]{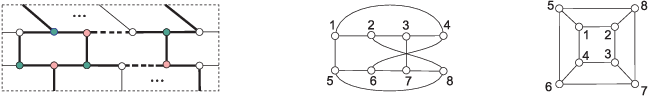}\\
{Figure \ref{fig4-5}: A $K_{3,3}$-subdivision in $K_e(k,2)$ and
$K_e(4,2)$ is isomorphic to $Q_3$.}
\end{center}
\end{figure}
Since $K_e(k,2)$ ($k\ge 6$) contains a subdivision of $K_{3,3}$ (see
Figure \ref{fig4-5}), it is non-planar. Note that $K_e(4,q)$ is
isomorphic to $T(4,q,1)$. Since $T(4,q,1)$ with $q\ge 3$ contains a
subdivision of $K_{3,3}$, it is non-planar. The polyhex graph
$K_e(4,2)$ is isomorphic to the cube and hence is planar. Hence $G$
is $K_e(4,2)$.
\end{proof}

By Proposition \ref{prop2-8}, Theorem \ref{thm2-6} and Lemmas
\ref{lem4-5} and \ref{lem4-6}, we have the following
characterization of Pfaffian bipartite polyhex graphs on the Klein
bottle.

\begin{thm}\label{thm4-7}
Let $G$ be a bipartite polyhex graphs on the Klein bottle. Then $G$
is Pfaffian if and only if it is isomorphic to the cube.
\end{thm}

Now, consider non-bipartite polyhex graphs on the Klein bottle. A
cycle $C$ of a graph $G$ on a surface is {\em 1-sided} if its
tubular neighborhood is homeomorphic to a M\"{o}bius strip, and {\em
2-sided}, otherwise. An embedding of a graph $G$ in the Klein bottle
is {\em cross-cap-odd} if every non-separating cycle $C$ of $G$ is
odd if and only if it is 1-sided.

\begin{lem}[\cite{Nor1}]\label{lem4-8}
Every graph that admits a cross-cap-odd embedding in the Klein
bottle is Pfaffian.
\end{lem}

\begin{figure}[!hbtp]\refstepcounter{figure}\label{fig4-6}
\begin{center}
\includegraphics[scale=1.4]{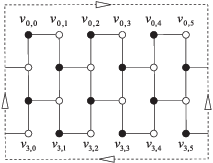}\\
{Figure \ref{fig4-6}: The tube $T_{4,6}$ obtained from $K_o(4,6)$ by
deleting edges $v_{0,i}v_{3,5-i}$ for $i\in \mathbb Z_{6}$.}
\end{center}
\end{figure}

\begin{lem}\label{lem4-9}
A non-bipartite polyhex graph $K_o(k,q)$ is a cross-cap-odd
embedding.
\end{lem}

\begin{proof} For a $K_o(k,q)$, let
\[
E_0 :=\left\{
 \begin{array}{ll}
 \{v_{0,i}v_{k-1,q-1-i}|i\in \mathbb Z_k\},     &\mbox{if $k\equiv 0$ (mod 2);}\\
 \{v_{0,i}v_{k-1,q-2-i}|i\in \mathbb Z_k\},   &\mbox{otherwise.}
 \end{array}
 \right.
\]
The subgraph obtained from $K_o(k,q)$ by deleting all edges in $E_0$
is a tube, denoted by $T_{k,q}$. Conversely, $K_o(k,q)$ can
generated from the tube $T_{k,q}$ by adding edges in $E_0$.

Note that $T_{k,q}$ is a bipartite graph since the proper 2-coloring
of $L(k,q)$ induces a proper 2-coloring of $T_{k,q}$. Any edge in
$E_0$ joins two vertices with the same color. Now for any cycle $C$
of $K_o(k,q)$, if $C$ is also a cycle of $T_{k,q}$, then $C$ is
2-sided and has even length. If $C$ is not a cycle of $T_{k,q}$, then
$E(C)\cap E_0\ne \emptyset$. Let $\delta:=|E(C)\cap E_0|$.
If contracting every edge in $E_0\cap E(C)$ in $C$ to a single vertex, 
we obtain a new cycle $C'$, which has  even length since the white vertices
and black vertices  alternate along any direction of $C'$. Hence
 $|E(C)|\equiv \delta$
(mod 2).

Note that the tubular neighborhood of $C$ is homeomorphic to a
M\"{o}bius strip if and only if $\delta=|E(C)\cap E_0|\equiv 1$
(mod 2). So $C$ is 1-sided if and only if $|E(C)|\equiv
\delta=|E(C)\cap E_0|\equiv 1$ (mod 2). That is, $K_o(k,q)$
is a cross-cap-odd embedding.
\end{proof}

By Lemmas \ref{lem4-8} and \ref{lem4-9}, the following result follows
immediately.

\begin{thm}\label{thm4-10}
Every non-bipartite polyhex graph on the Klein bottle is Pfaffian.
\end{thm}

In \cite{LLZ}, it has been shown that $K_o(k,q)$ is 2-extendable
if and only if $k\ge 4$ and $q\ge 5$. By a result of Lov\'asz and
Plummer (Theorem 5.5.23 on Page 206 in \cite{LP}) that a 2-extendable
graph is either bicritical or elementary bipartitie, $K_o(k,q)$ is bicritical 
and hence a brick. By Theorem~\ref{thm:4-8}, the face-width
$\text{fw}(K_o(k,q))=\min\{\lceil \frac k 2\rceil, q\}\to \infty$ as 
$\min\{k,q\}\to \infty$.
Hence we have the following remark.

\begin{rem}Theorem \ref{thm:2-3} shows that a Pfaffian
brace embedded on a surface $\Sigma$ with $g(\Sigma)>0$ has a small
face-width. But the face-width of a Pfaffian brick on sufaces with positive genus could be arbitrarily large.
\end{rem}

\begin{rem} In \cite{NT}, Norine and Thomas conjectured
that every Pfaffian cubic graph is 3-edge colorable. As shown in
\cite{Y}, every polyhex graph on the Klein bottle is Hamiltonian and
hence is 3-edge colorable. By Theorem \ref{thm4-10}, every Pfaffian
polyhex graph is 3-edge colorable. Hence  the conjecture of
Norine and Thomas is true for Pfaffian
polyhex graphs.
\end{rem}

\begin{thebibliography}{00}
\parskip=-0.1cm

\bibitem{A} A. Altschuler, Construction and enumeration of regular
maps on the torus, Discrete Math. 4 (1973) 201-217.


\bibitem{DFRR} M. Deza, P.W. Fowler, A. Rassat and K.M. Rogers,
Fullerenes as tilings of surfaces, J. Chem. Inf. Comput. Sci. 40
(2000) 550-558.


\bibitem{HP} D. Holton and M.D. Plummer, 2-extendability in
3-polytopes, In: {\em Combinatorics}, Eger, Hungary, 1987, Colloq.
Math. Soc. J. Bolyai, Vol. 52, Akad\'emiai Kiad\'o, Budapest, 1988,
pp. 281-300.

\bibitem{HP2} D. Holton and M.D. Plummer, Matching extension and
connectivity in graphs II, In: {\em Graph theory, Combinatorics, and
Applications}, Vol. 2 (Kalamazoo, MI, 1988), Wiley, New York, 1991,
pp. 651-665.

\bibitem{J} P.E. John, Kekul\'e count in toroidal hexagonal carbon cages,
Croat. Chem. Acta 71 (3) (1998) 435-447.

\bibitem{K1} P.W. Kasteleyn, The statistics of dimers on a lattice.
I. The number of dimer arrangments on a quadratic lattice, Physica
27 (1961) 1209-1225.

\bibitem{K2} P.W. Kasteleyn, Graph theory and crystal physics, In:
{\em Graph Theory and Theoretical Physics} (F. Harary eds.),
Academic Press, NewYork 1967, pp. 43-110.

\bibitem{KI} E.C. Kirby, Recent work on toroidal and other exotic
fullerene structures, in: {\em From Chemical Topology to
Three-Dimensional Geometry} (A.T. Balaban eds.), Plenum Press, New
York, 1997, pp. 263-296.

\bibitem{KMP} E.C. Kirby, R.B. Mallion and P. Pollak, Toroidal
polyhexes, J. Chem. Soc. Faraday Trans. 89(12) (1993) 1945-1953.


\bibitem{LLZ} Q. Li, S. Liu and H. Zhang, 2-extendability and $k$-resonance
of non-bipartite Klein-bottle polyhexes, Discrete Appl. Math. 159
(8) (2011) 800-811.

\bibitem{Little} C.H.C. Little,  A characterization of convertible (0, 1)-matrices. J. Combin. Theory Ser.
B 18 (1975), 187-208.

\bibitem{L} L. Lov\'asz, Matching structure and the matching lattice, J.
Combin. Theory Ser. B 43 (1987) 187-222.

\bibitem{LP} L. Lov\'asz and M.D. Plummer, {\em Matching Theory}, Ann.
Discrete Math. 29, North-Holland, Amsterdam, 1986.


\bibitem{WM} W. McCuaig, P\'olya's permanent problem,
Electron. J. Combin. 11 (2004) \#R79, 83pp.

\bibitem{MT} B. Mohar and C. Thomassen, {\em Graphs on Surfaces}, Johns
Hopkins University Press, Baltimore, 2001.

\bibitem{Neg} S. Negami, Uniqueness and faithfulness of embedding of
toroidal graphs, Discrete Math. 44 (1983) 161-180.

\bibitem{Nor1} S. Norine, Matching structure and Pfaffian orientations
of graphs, PhD Dissertation, Georgia Institute of Technology, 2005.



\bibitem{NT} S. Norine and R. Thomas, Pfaffian labelings and signs of
edge-colorings, Combinatorica 28 (1) (2008) 99-111.


\bibitem{P1} M.D. Plummer, On $n$-extendable graphs, Discrete Math.
31 (1980) 201-210.

\bibitem{P2} M.D. Plummer, Matching extension in bipartite graphs,
Congr. Numer. 54 (1986) 245-258.


\bibitem{RST} N. Robertson, P.D. Seymour and R. Thomas,
Permanenets, Pfaffian orientations, and even directed circuits,
Ann. Math. 150 (1999) 929-975.

\bibitem{RV90} N. Robertson and R.P. Vitray, Representativity of
surface embedding, in: {\em Paths, Flows, and VLSI-Layout} (B.
Krote, L. Lov\'asz, H.J. Pro\"{o}mel, and A. Schrijver eds.),
Springer-Verlag, Berlin, 1990, pp. 293-328.

\bibitem{SLZ} W.C. Shiu, P.C.B. Lam and H. Zhang, $k$-resonance in toroidal
polyhexes, J. Math. Chem. 38 (4) (2005) 451-466.


\bibitem{RT} R. Thomas, A survey of Pfaffian orientations of graphs,
In: {\em Proceedings of the International Congress of
Mathematicians}, Vol. 3, Madrid, 2006, pp. 963-984.

\bibitem{T90} C. Thomassen, Embeddings of graphs with no short noncontractible
cycles, J. Combin. Theory Ser. B 48 (1990) 155-177.

\bibitem{T91} C. Thomassen, Tilings of the torus and the Klein bottle and
vertex-transitive graphs on a fixed surface, Trans. Amer. Math. Soc.
323 (1991) 605-635.

\bibitem{VY} V.V. Vazirani and M. Yannakakis, Pfaffian orientations, 0-1
permanents, and even cycles in directed graphs, Discrete Appl. Math.
25 (1989) 179-190.

\bibitem{Y} D. Ye, Hamilton cycles in cubic polyhex graphs on the Klein
bottle, Ars Combin., in press.

\bibitem{YZ} D. Ye and H. Zhang, 2-extendability of toroidal
polyhexes and Klein-bottle polyhexes, Discrete Appl. Math. 157
(2009) 292-299.

\bibitem{ZY} H. Zhang and D. Ye, $k$-resonant toroidal polyhexes, J.
Math. Chem. 44 (1) (2008) 270-285.

\end{thebibliography}
\end{document}